\documentclass[10pt,twocolumn,letterpaper]{article}

\usepackage{cvpr}
\usepackage{times}
\usepackage{epsfig}
\usepackage{graphicx}
\usepackage{amsmath}
\usepackage{amssymb}
\usepackage{amsthm}
\usepackage[]{algorithm2e}

\newtheorem{theorem}{Theorem}

\newtheorem{lemma}{Lemma}

\def\C{{\mathcal{C}}}

\DeclareMathOperator*{\argmin}{arg\,min}
\DeclareMathOperator*{\argmax}{arg\,max}
\DeclareMathOperator{\rank}{rank}

\newcommand{\sing}[1]{\vec{\sigma}{#1}}

\newcommand{\inner}[2]{\langle#1,#2\rangle}
\newcommand{\frobnorm}[1]{\norm{#1}_F}
\DeclareMathOperator{\tr}{tr}
\newcommand{\te}[1]{\text{#1}}
\newcommand{\reg}{\ensuremath{\mathcal{R}}}

\usepackage{booktabs}
\usepackage[table]{xcolor}
\usepackage{array,multirow,graphicx}
\usepackage{float}

\makeatletter
\@namedef{ver@everyshi.sty}{} 
\makeatother
\usepackage{tikz,pgfplots}
\pgfplotsset{compat=1.3}

\usepackage{algorithmic}
\SetAlCapNameFnt{\small}
\SetAlCapFnt{\small}
\SetAlCapSty{}

\usepackage[pagebackref=true,breaklinks=true,letterpaper=true,colorlinks,bookmarks=false]{hyperref}

\cvprfinalcopy 

\def\cvprPaperID{9438} 

\def\reg{\mathcal{R}}
\def\A{\mathcal{A}}

\newcommand{\fr}[2]{\frac{#1}{#2}}

\usepackage{bm}
\renewcommand{\vec}[1]{%
  \ifcat\noexpand#1\relax 
    \bm{#1} 
  \else
    \mathbf{#1} 
  \fi
}
\newcommand{\hada}{\odot}
\usepackage{bbm}
\newcommand{\R}{\mathbbm{R}}

\usepackage{mathtools}
\newcommand{\norm}[1]{\left\|#1\right\|}
\newcommand{\T}{T}

\ifcvprfinal\fi
\begin{document}

\title{A Unified Optimization Framework for Low-Rank Inducing Penalties}

\def\ww{5mm}
\author{Marcus Valtonen \"Ornhag$^1$ \hspace{\ww} Carl Olsson$^{1,2}$ \hspace{\ww} Anders Heyden$^1$\\[0.2 cm]
	\begin{minipage}[c]{0.4\textwidth}
		\centering
		${}^1$Centre for Mathematical Sciences\\
		Lund University 
	\end{minipage}
	\begin{minipage}[c]{0.4\textwidth}
	\centering
	${}^2$Department of Electrical Engineering\\
	Chalmers University of Technology 
\end{minipage}
	\\[0.4cm]
	{\tt\small  \{marcus.valtonen\_ornhag,\,carl.olsson,\,anders.heyden\}@math.lth.se}
}

\maketitle

\begin{abstract}
In this paper we study the convex envelopes of a new class of functions. Using
this approach, we are able to unify two important classes of regularizers from
unbiased non-convex formulations and weighted nuclear norm penalties.
This opens up for possibilities of combining the best of both worlds, and to leverage
each methods contribution to cases where simply enforcing one of the regularizers
are insufficient.

We show that the proposed regularizers can be incorporated in standard splitting schemes
such as
Alternating Direction
Methods of Multipliers (ADMM), and other subgradient methods. Furthermore, we provide an
efficient way of computing the
proximal operator.

Lastly, we show on real non-rigid structure-from-motion (NRSfM) datasets, the issues
that arise from using weighted nuclear norm penalties, and how this can be
remedied using our proposed method.
\end{abstract}

\section{Introduction}
Dimensionality reduction using Principal Component Analysis (PCA) is widely used
for all types of data analysis, classification and clustering. In recent years,
numerous subspace clustering methods have been proposed, to address the shortcomings
of traditional PCA methods. The work on Robust PCA by Cand{\`e}s~\etal{}~\cite{candes2011robust}
is one of the most influential papers on the subject, which sparked a large research
interest from various fields including computer vision. Applications include, but are not
limited to,
rigid and non-rigid structure-from-motion~\cite{bregler-etal-cvpr-2000,angst-etal-iccv-2011},
photometric stereo~\cite{basri-etal-ijcv-2007} and optical flow~\cite{garg-etal-ijcv-2013}.

It is well-known that the solution to
\begin{equation}
    \min_{\rank(X)\leq r}\frobnorm{X-X_0}^2,
\end{equation}
where~$\frobnorm{\cdot}$ is the Frobenius norm,
can be given in closed form using the singular value decomposition (SVD) of the
measurement matrix~$X_0$.
The character of the problem changes drastically, when
considering objectives such as
\begin{equation}
    \label{eq:hardrank}
    \min_{\rank(X)\leq r}\norm{\A(X)-\vec{b}}^2,
\end{equation}
where~$\A\::\:\R^{m\times n}\rightarrow\R^p$ is a linear operator, $\vec{b}\in\R^{p}$,
and~$\norm{\cdot}$ is the standard Euclidean norm. In fact, such problems are in general
known to be NP hard~\cite{gillis-glineur-siam-2011}.
In many cases, however, the rank is not known \emph{a priori},
and a ``soft rank'' penalty can be used instead
\begin{equation}
    \label{eq:softrank}
    \min_X    \mu\rank(X)+\norm{\A(X)-\vec{b}}^2.
\end{equation}
Here, the regularization parameter~$\mu$ controls the trade-off between enforcing
the rank and minimizing the residual error, and can be tuned to problem specific applications.

In order to treat objectives of the form~\eqref{eq:hardrank} and~\eqref{eq:softrank},
a convex surrogate of the rank penalty is often used. One popular approach
is to use the nuclear norm~\cite{recht-etal-siam-2010,candes2011robust}
\begin{equation}
    \norm{X}_* = \sum_{i=1}^n\sigma_i(X),
\end{equation}
where $\sigma_i(X)$, $i=1,\ldots,n$, is the $i$:th singular value of~$X$.
One of the drawbacks of using the nuclear norm penalty is that both large and
small singular values are penalized equally hard. This is referred to as shrinking bias,
and to counteract such behavior, methods penalizing small singular values (assumed
to be noise) harder have been proposed~\cite{oymak-etal-2015,mohan2010iterative,hu-etal-pami-2013,oh-etal-pami-2016,olsson-etal-iccv-2017,larsson-olsson-ijcv-2016,canyi2015,shang-etal-2018}.

\subsection{Related Work}
Our work is a generalization of Larsson and Olsson~\cite{larsson-olsson-ijcv-2016}. They
considered problems on the form
\begin{equation}
    \label{eq:gprob}
    \min_X g(\rank(X))+\frobnorm{X-X_0}^2,
\end{equation}
where the regularizer~$g$ is non-decreasing and piecewise constant,
\begin{equation}
    \label{eq:g}
    g(k) = \sum_{i=1}^kg_i.
\end{equation}
Note, that for $g_i\equiv\mu$ we obtain~\eqref{eq:softrank}. Furthermore, if we let
$g_i=0$ for $i\leq r_0$, and $\infty$ otherwise, \eqref{eq:hardrank} is obtained.
The objectives~\eqref{eq:gprob} are difficult to optimize, as they, in general, are non-convex and discontinuous. Thus, it is natural to consider a relaxation
\begin{equation}
    \min_X \reg_g(X)+\frobnorm{X-X_0}^2,
\end{equation}
where
\begin{equation}
    \reg(X) = \max_{Z}\left(\sum_{i=1}^n\min(g_i,\sigma_i^2(Z))-\frobnorm{X-Z}^2\right).
\end{equation}
This is the convex envelope of~\eqref{eq:gprob}, hence share the same global minimizers.

Another type of regularizer that has been successfully used in low-level imaging
applications~\cite{gu-2016,mswnnm,mcwnnm}
is the weighted nuclear norm (WNNM),
\begin{equation}
    \norm{X}_{\vec{w},*} = \sum_{i=1}^k w_i\sigma_i(X),
\end{equation}
where~$\vec{w}=(w_1,\ldots, w_k)$ is a weight vector.
Note that the WNNM formulation does not fit the assumptions~\eqref{eq:g}, hence
cannot be considered in this framework.

For certain applications, it is of interest
to include both regularizers, which we will show in Section~\ref{sec:experiments}.
Typically, this is preferable when the rank constraint alone is not strong enough to
yield accurate reconstructions, and further penalization is necessary to restrict
the solutions. To this end, we suggest to merge these penalties.
Our main contributions are:
\begin{itemize}
    \item A novel method for combining bias reduction and non-convex low-rank inducing objectives,
    \item An efficient and fast algorithm to compute the proposed regularizer,
    \item Theoretical insight in the quality of reconstructed missing data using WNNM, and
          practical demonstrations on how shrinking bias is perceived in these applications,
    \item A new objective for Non-Rigid Structure from Motion (NRSfM), with improved performance,
          compared to state-of-the-art prior-free methods, capable of working in cases where
          the image sequences are unordered.
\end{itemize}

First, however, we will lay the ground for the theory of a common framework
of low-rank inducing objectives.

\section{Problem Formulation and Motivation}
In this paper we propose a new class of regularization terms for low rank matrix recovery problems that controls both the rank and the size of the singular values of the recovered matrix.
Our objective function has the form
\begin{equation}
f_h(X) = h(\sing(X)) + \|\A (X) - b\|^2,
\end{equation}
where $h(\sing(X)) = \sum_{i=1}^k h_i(\sigma_i(X))$ and 
\begin{equation}
h_i(\sigma_i(X)) = 
\begin{cases}
2a_i \sigma_i (X) + b_i & \sigma_i(X) \neq 0, \\
0 & \text{otherwise}.
\end{cases}
\end{equation}
We assume that the sequences $\{a_i\}_{i=1}^k$ and $\{b_i\}_{i=1}^k$ are both non-decreasing.

Our approach unifies the formulation of \cite{larsson-etal-eccv-2014} with weighted nuclear norm.
The terms~$2a_i \sigma_i(X)$ correspond to the singular value penalties of a weighted nuclear norm~\cite{gu-2016}. 
These can be used to control the sizes of the non-zero singular values. In contrast, the constants $b_i$ corresponds to a rank penalization that is independent of these sizes and, as we will see in the next section, enables bias free rank selection. 

\subsection{Controlled Bias and Rank Selection}
To motivate the use of both sets of variables $\{a_i\}_{i=1}^k$ and $\{b_i\}_{i=1}^k$, and to understand their purpose, we first consider the simple recovery problem~$\min_Xf_h(X)$, where
\begin{equation}\label{eq:fh}
f_h(X) \coloneqq h(\sing(X)) + \|X-X_0\|_F^2.
\end{equation}
Here~$X_0$ is assumed to consist of a set of large singular values $\sigma_i(X_0)$, $i=1,...,r$, corresponding to the matrix we wish to recover, and a set of small ones  $\sigma_i(X_0)$, $i=r+1,...,k$, corresponding to noise that we want to suppress.

Due to von Neumann's trace theorem the solution can be computed in closed form by considering each singular values separately, and minimize
\begin{equation}
\begin{cases}
2a_i s + b_i + (s-\sigma_i(X_0))^2 & s \neq 0, \\
\sigma_i(X_0)^2 & s = 0,
\end{cases}
\end{equation}
over $s\geq 0$. Differentiating for the case $s \neq 0$ gives a stationary point at
$s = \sigma_i(X) - a_i$ if $\sigma_i(X) - a_i > 0$.
Since this point has objective value $2a_i \sigma_i(X_0) - a_k^2 + b_k$ it is clear that this point will be optimal if
\begin{equation}
2a_i \sigma_i(X_0) - a_i^2 + b_i \leq \sigma_i(X_0)^2,
\end{equation}
or equivalently
$\sigma_i(X_0) - a_i \geq \sqrt{b_i}.$
Summarizing, we thus get the optimal singular values
\begin{equation}
\sigma_i(X) = \begin{cases}
\sigma_i(X_0) - a_i & \sigma_i(X_0) - a_i \geq \sqrt{b_i}, \\
0 & \text{otherwise}.
\end{cases}
\end{equation}
Note, that this is a valid sequence of singular values since under our assumptions $\sigma_i(X_0) - a_i$ is decreasing and $\sqrt{b_i}$ increasing. The red curve of Figure~\ref{fig:bias} shows the recovered singular value as a function of the corresponding observed singular value. For comparison, we also plot the dotted blue curve which shows hard thresholding at $a_i + \sqrt{b}_i$,~\ie{} singular values smaller than $a_i + \sqrt{b}_i$ vanish while the rest are left unaltered.
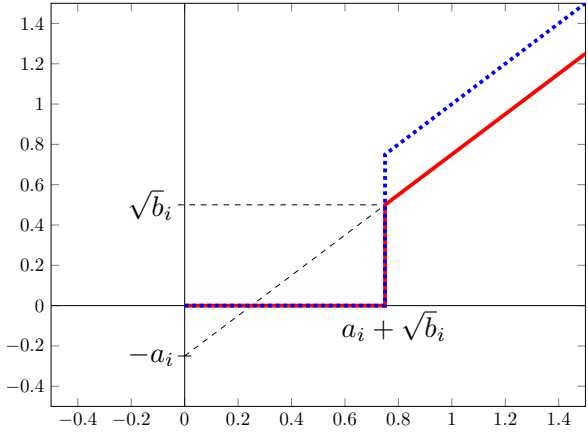
\begin{figure}
\centering
\resizebox{78mm}{!}{
%
%
\begin{tikzpicture}

\begin{axis}[%
width=4.521in,
height=3.566in,
at={(0.758in,0.481in)},
xmin=-0.5,
xmax=1.5,
ymin=-0.5,
ymax=1.5,
axis background/.style={fill=white}
]
\addplot [color=black, forget plot]
  table[row sep=crcr]{%
-0.5	0\\
-0.4	0\\
-0.3	0\\
-0.2	0\\
-0.1	0\\
0	0\\
0.1	0\\
0.2	0\\
0.3	0\\
0.4	0\\
0.5	0\\
0.6	0\\
0.7	0\\
0.8	0\\
0.9	0\\
1	0\\
1.1	0\\
1.2	0\\
1.3	0\\
1.4	0\\
1.5	0\\
};
\addplot [color=black, forget plot]
  table[row sep=crcr]{%
0	-0.5\\
0	-0.4\\
0	-0.3\\
0	-0.2\\
0	-0.1\\
0	0\\
0	0.1\\
0	0.2\\
0	0.3\\
0	0.4\\
0	0.5\\
0	0.6\\
0	0.7\\
0	0.8\\
0	0.9\\
0	1\\
0	1.1\\
0	1.2\\
0	1.3\\
0	1.4\\
0	1.5\\
};
\addplot [color=black, dashed, forget plot]
  table[row sep=crcr]{%
-0.025	0.5\\
0.75	0.5\\
};
\addplot [color=black, forget plot]
  table[row sep=crcr]{%
-0.025	-0.25\\
0.025	-0.25\\
};
\addplot [color=black, dashed, forget plot]
  table[row sep=crcr]{%
0	-0.25\\
1.5	1.25\\
};
\addplot [color=red, line width=2.0pt, forget plot]
  table[row sep=crcr]{%
0	0\\
0.25	0\\
0.75	0\\
0.75	0.5\\
1.5	1.25\\
};
\addplot [color=blue, dotted, line width=2.0pt, forget plot]
  table[row sep=crcr]{%
0	0\\
0.25	0\\
0.75	0\\
0.75	0.75\\
1.5	1.5\\
};
\node[right, align=left,scale=1.5]
at (axis cs:-0.25,0.5) {$\sqrt{b}_i$};
\node[right, align=left,scale=1.5]
at (axis cs:-0.25,-0.25) {$-a_i$};
\node[right, align=left,scale=1.5]
at (axis cs:0.55,-0.1) {$a_i+\sqrt{b}_i$};
\end{axis}
\end{tikzpicture}%
}
\caption{The optimal recovered singular value $\sigma_i(X)$ as a function (red curve) of the observed $\sigma_i(X_0)$.}
\label{fig:bias}
\vspace{-4.5mm}
\end{figure}

Now, suppose that we want to recover the largest singular values unchanged. 
Using the weighted nuclear norm (\mbox{$b_i = 0$}) it is clear that this can only be done if we know that the sought matrix has rank $r$ and let $a_i = 0$ for $i=1,...,r$. 
For any other setting the regularization will subtract $a_i$ from the corresponding non-zero singular value.
In contrast, by letting $a_i=0$ allows exact recovery of the large singular values by selecting $b_i$ appropriately even when the rank is unknown. Hence, in the presence of a weak prior on the rank of the matrix, using only the $b_i$ (the framework in \cite{larsson-olsson-ijcv-2016}) allows exact recovery for a more general set of problems than use of the $a_i$ (weighted nuclear norm formulations).

The above class of problems are well posed with a strong data term $\|X-X_0\|_F^2$.
For problems with weaker data terms, priors on the size of the singular values can still be very useful.
In the context of NRSfM it has been observed \cite{olsson-etal-iccv-2017,dai-etal-ijcv-2014} that adding a bias can improve the distance to the ground truth reconstruction, even though it does not alter the rank.
The reason is that, when the scene is not rigid, several reconstructions with the same rank may co-exist, thus resulting in similar projections. By introducing bias on the singular values, further regularization is enforced on the deformations, which may aid in the search for correct reconstructions.
For example, with $a_1 = 0$ and $a_i > 0$, $i > 1$ we obtain a penalty that favors matrices that "are close to" rank $1$. In the formulation of \cite{dai-etal-ijcv-2014}, where rank~$1$ corresponds to a rigid scene this can be thought of as an "as rigid as possible" prior, which is realistic in many cases \cite{videopopup,dynamicfusion}, but which has yet to be considered in the context of factorization methods.
\footnote{To regularize the problem Dai~\etal{} incorporated a penalty of the derivatives of the 3D tracks, which also can be seen as a prior preferring rigid reconstructions. However, this option is not feasible for unsorted image collections.}

\subsection{The Quadratic Envelope}
As discussed above the two sets of parameters $\{a_i\}$ and $\{b_i\}$ have complementary regularization effects. 
The main purpose of unifying them is to create more flexible priors allowing us to do accurate rank selection with a controlled bias. In the following sections, we also show that they have relaxations that can be reliably optimized.
Specifically, the resulting formulation $h(\sing{(X)})$, which is generally non-convex and discontinuous, can be relaxed by computing the so called quadratic envelope \cite{carlsson2016convexification}.
The resulting relaxation $\reg_h(\sing(X))$ is continuous and in addition $\reg_h(\sing(X))+\|X-X_0\|_F^2$ is convex. 
For a more general data term there may be multiple local minimizers. 
However, it is known that 
\begin{equation}
    h(\sing(X))+\|\A (X)-b\|^2,
\label{eq:unrelaxed}
\end{equation}
and 
\begin{equation}
    \reg_h(\sing(X))+\|\A (X)-b\|^2,
\label{eq:relaxed}
\end{equation}
have the same global minimizer when $\|\A \| < 1$ \cite{carlsson2016convexification}. In addition, potential local minima of \eqref{eq:relaxed} are also local minima of \eqref{eq:unrelaxed}; however, the converse does not hold.
We also show that the proximal operator of $\reg_h(\sing(X))$ can be efficiently computed which allows simple optimization using splitting methods such as ADMM \cite{boyd-etal-2011}.

\section{A New Family of Functions}
Consider functions on the form~\eqref{eq:fh}.
This is a generalization of~\cite{larsson-olsson-ijcv-2016}; and the derivation for
our objective follows a similar structure. We outline this in detail in the
supplementary material, where we show that
convex envelope~$f^{**}_h$ is given by
\begin{equation}
    f^{**}_h(X) = \reg_h(X)+\frobnorm{X-X_0}^2,
\end{equation}
where
\begin{equation}\label{eq:reghfull}
\begin{aligned}
    \reg_h(X)
    &\coloneqq
    \max_Z\Bigg(\sum_{i=1}^n\min\left(b_i, \left[\sigma_i(Z)-a_i\right]_+^2\right)+ \frobnorm{Z}^2\\
    &-\frobnorm{X-Z}^2  - \sum_{i=1}^n\left[\sigma_i(Z)-a_i\right]_+^2\Bigg).
\end{aligned}
\end{equation}
Furthermore, the optimization can be reduced to the singular values only,
\begin{equation}
\begin{aligned}
\label{eq:maxsingz}
    \reg_h(X)
    &=
    \max_{\sing(Z)}
    \Bigg(
    \sum_{i=1}^n
    \min\left(b_i, \left[\sigma_i(Z)-a_i\right]_+^2\right)
    +
    \sigma_i^2(Z) \\
    &-
    (\sigma_i(X)-\sigma_i(Z))^2
    -
    \left[\sigma_i(Z)-a_i\right]_+^2\Bigg).
\end{aligned}
\end{equation}
This optimization problem is concave, hence can be solved with standard convex solvers
such as MOSEK or CVX; however, in the next section we show that the problem
can be turned into a series of one-dimensional problems, and the resulting algorithm
for computing~\eqref{eq:reghfull} is magnitudes faster than applying a general purpose solver.

\section{Finding the Maximizing Sequence}\label{sec:finding}
Following the approach used in~\cite{larsson-olsson-ijcv-2016},
consider the program
\begin{equation}
\label{eq:unconstrained}
\begin{aligned}
&\max_s    && f(s) \\
&\te{s.t.} && \sigma_{i+1}(Z)\leq s\leq \sigma_{i-1}(Z).
\end{aligned}
\end{equation}
where~$\sigma_i(Z)$ is the $i$:th singular value of the maximizing sequence
in~\eqref{eq:maxsingz}, and
\begin{equation}\label{eq:f}
f(s) = \min\{b_i,[s-a_i]_+^2\}-(s-\sigma_i(X))^2+s^2-[s-a_i]_+^2.
\end{equation}
The objective function~$f$ can be seen as the
pointwise minimum of two concave functions, namely,
$f_1(s) = b_i+2\sigma_i(X)s-\sigma_i^2(X)-[s-a_i]_+^2$ and
$f_2(s) = 2\sigma_i(X)s-\sigma_i(X)^2$,~\ie{}
$f(s) = \min\{f_1(s),f_2(s)\}$, hence~$f$ is concave.

The individual unconstrained optimizers are given
by $s_i=a_i+\max\{\sqrt{b_i},\sigma_i(X)\}$. In previous work~\cite{larsson-olsson-ijcv-2016},
where $a_i\equiv 0$, an algorithm was devised to find the maximizing singular
vector, by turning it to a optimization problem of a single variable. This
method is not directly applicable, as the sequence~$\{s_i\}_{i=1}^k$, in general,
does not satisfy the necessary conditions\footnote{In order to use the algorithm,
the sequence~$\{s_i\}_{i=1}^k$ must be non-increasing for $i<p$,
non-decreasing for $p\leq i\leq q$, and non-increasing for $i>q$, for some $p$, and $q$.}.
In fact, the number of local extrema in the sequence~$\{s_i\}_{i=1}^k$ is only limited by its
length. We show an example of a sequence in Figure~\ref{fig:experimental},
and the corresponding maximizing sequence.
Nevertheless, it is possible to devise an algorithm that returns the
maximizing singular value vector, as we will show shortly.

In order to do so, we can apply some of the thoughts behind the proof
behind~\cite{larsson-olsson-ijcv-2016}.
Consider the more general optimization problem of
minimizing~$g(\vec{\sigma})=\sum_{i=1}^kf_i(s_i)$,
subject to $\sigma_1\geq\sigma_2\geq\cdots\geq \sigma_k\geq 0$, where $f_i$ are
concave. Then, given the unconstrained sequence of minimizers $\{s_i\}_{i=1}^k$, the elements of the
constrained
sequence~$\{\sigma_i\}_{i=1}^k$
can be limited to three choices
\begin{equation}
\label{eq:cases}
\sigma_i =
\left\{
\begin{aligned}
    & s_i && \te{ if } \sigma_{i+1}\leq s_i \leq \sigma_{i-1}, \\
    & \sigma_{i-1} && \te{ if } \sigma_{i-1} < s_i, \\
    & \sigma_{i+1} && \te{ if }  s_i < \sigma_{i+1}. \\
\end{aligned}
\right.
\end{equation}
Furthermore, the regions between local extreme points (of the unconstrained singular values)
are constant.

\begin{lemma}
\label{lem:1}
Assume $s_p$ and $s_q$ are local extrema of $\{s_i\}_{i=1}^k$ and that the subsequence
$\{s_i\}_{i=p}^q$ are non-decreasing. Then the corresponding subsequence of the
constrained problem $\{\sigma_i\}_{i=p}^q$ is constant.
\end{lemma}
\begin{proof}
Consider $\sigma_i$ for some $p\leq i\leq q-1$. If $\sigma_i>s_i$, then by~\eqref{eq:cases}
we have $\sigma_{i+1}=\sigma_i$. If instead $\sigma_i\leq s_i$, we have
$\sigma_{i+1}\leq\sigma_i\leq s_i\leq s_{i+1}$ and by~\eqref{eq:cases}, $\sigma_{i+1}=\sigma_i$.
\end{proof}

\begin{algorithm}[t]
\KwData{Weights $\vec{a}$, $\vec{b}$, and $\sing(X)$}
\KwResult{Maximizing vector $\sing(Z^*)=\{\sigma_i\}_i$}
Initialize with the unconstrained maximizers~$\sigma_i = a_i + \max\{\sqrt{b_i},\sigma_i(X)\}$\;
\While{$\sing(Z^*)$ is not a valid singular value vector}{
    Find local extrema of $\sing(Z^*)$ and generate subintervals $\{\iota_k\}_{k\in\mathcal{I}}$\;
    \For{$k\in\mathcal{I}$}{
        Find scalar~$s^* = \argmax_sf(s)$ where $f$ is defined in~\eqref{eq:f}\;
        Update $\sigma_i = s^*$ for all $i\in\iota_k$.
    }
}
\caption{Algorithm for finding the maximizing singular value vector.}
\label{algo:1}
\end{algorithm}

We can now devise an algorithm that returns the maximizing sequence, see Algorithm~\ref{algo:1}.
Essentially, the algorithm starts at the unconstrained solution, and then
adds more constraints, by utilizing Lemma~\ref{lem:1}, until all of them are fulfilled.

\begin{theorem}\label{thm:algo1}
Algorithm~\ref{algo:1} returns the maximizing sequence.
\end{theorem}

\begin{proof}
See the supplementary material.
\end{proof}

\begin{figure}[h]
\centering
\includegraphics[width=0.95\linewidth]{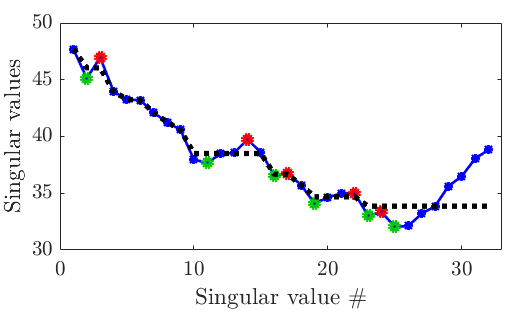}%
\caption{Example of a sequence of unconstrained maximizers (blue line), local extrema (green and red) and the maximizing sequences (dashed black) obtained by Algorithm~\ref{algo:1}.}
\label{fig:experimental}
\end{figure}

\newcommand{\STAB}[1]{\begin{tabular}{@{}c@{}}#1\end{tabular}} 

\begin{table*}[t]
\centering
\caption{Distance to ground truth (normalized) mean valued over 20 problem instances for
different percentages of missing data and data patterns. The standard deviation of the noise
is kept constant at~$\sigma=0.1$. Best results are marked in bold.}
\setlength\tabcolsep{0.14cm}
{\footnotesize
\rowcolors{2}{blue!15}{white!10}
\begin{tabular}{r | rrrrrr | rr | rr | r}

Missing\\data (\%) & PCP~\cite{candes-etal-acm-2011} & WNNM~\cite{gu-2016} &
Unifying~\cite{cabral-etal-iccv-2013} & LpSq~\cite{nie-2012} &
S12L12~\cite{shang-etal-2018} & S23L23~\cite{shang-etal-2018} &
IRNN~\cite{canyi2015} & APGL~\cite{toh-yun-2010} &
$\norm{\cdot}_*$~\cite{boyd-etal-2011}  & $\reg_\mu$~\cite{larsson-olsson-ijcv-2016} & Our\\

\toprule 
0 & 0.0400 & 0.0246 & 0.0406 & 0.0501 & 0.0544 & 0.0545 & 0.0551 & 0.0229 & 0.1959 & \textbf{0.0198} & 0.0199 \\
20 & 0.3707 & 0.2990 & 0.3751 & 0.1236 & 0.1322 & 0.0972 & 0.0440 & 0.0233 & 0.2287 & 0.0257 & \textbf{0.0198} \\
40 & 1.0000 & 0.6185 & 0.9355 & 0.1265 & 0.1222 & 0.1137 & 0.0497 & 0.0291 & 0.3183 & 0.2105 & \textbf{0.0248} \\
60 & 1.0000 & 0.8278 & 1.0000 & 0.1354 & 0.1809 & 0.1349 & 0.0697 & 0.0826 & 0.5444 & 0.3716 & \textbf{0.0466} \\
\multirow{-6}{*}{\STAB{\rotatebox[origin=c]{90}{\scriptsize Uniform}}\phantom{AA}}
80 & 1.0000 & 0.9810 & 1.0000 & 0.7775 & 0.6573 & 0.5945 & \textbf{0.2305} & 0.4648 & 0.8581 & 0.9007 & 0.3117 \\

\midrule 
0 & 0.0399 & 0.0220 & 0.0399 & 0.0491 & 0.0352 & 0.0344 & 0.0491 & 0.0205 & 0.1762 & \textbf{0.0176} & 0.0177 \\
10 & 0.3155 & 0.2769 & 0.1897 & 0.1171 & 0.0881 & 0.0874 & 0.0926 & 0.1039 & 0.2607 & 0.0829 & \textbf{0.0802} \\
20 & 0.4681 & 0.4250 & 0.3695 & 0.1893 & 0.1346 & \textbf{0.1340} & 0.1430 & 0.1686 & 0.3425 & 0.2146 & 0.1343 \\
30 & 0.5940 & 0.5143 & 0.4147 & 0.1681 & 0.2822 & 0.3081 & 0.1316 & 0.1594 & 0.3435 & 0.4137 & \textbf{0.1277} \\
40 & 0.7295 & 0.6362 & 0.9331 & 0.2854 & 0.4262 & 0.4089 & 0.1731 & 0.2800 & 0.5028 & 0.5072 & \textbf{0.1705} \\
\multirow{-6}{*}{\STAB{\rotatebox[origin=c]{90}{\scriptsize Tracking}}\phantom{AA}}
50 & 0.7977 & 0.7228 & 0.9162 & 0.4439 & 0.5646 & 0.5523 & \textbf{0.2847} & 0.4219 & 0.5831 & 0.6464 & 0.3128 \\
\bottomrule
\end{tabular}
\label{tab:synth}
}
\end{table*}

\section{ADMM and the Proximal Operator}
We employ the splitting method ADMM~\cite{boyd-etal-2011}, which is a standard tool for
problems of this type. Thus, consider the augmented Lagrangian
\begin{equation}
L(X,Y,\Lambda) = f_h^{**}(X)+\rho\frobnorm{X-Y+\Lambda}^2+\C(Y)-\rho\frobnorm{\Lambda}^2.
\end{equation}
In each iteration, we solve
\begin{align}
X_{t+1} &=\argmin_Xf_h^{**}(X)+\rho\frobnorm{X-Y_t+\Lambda_t}^2,\\
Y_{t+1} &=\argmin_Y\rho\frobnorm{X_{t+1}-Y+\Lambda_t}^2+\C(Y),\\
\Lambda_{t+1} &=X_{t+1} - Y_{t+1} + \Lambda_t.
\end{align}
To evaluate the proximal operator~$f_h^{**}$ one must solve
\begin{equation}
\min_X\reg_h(X) + \frobnorm{X-X_0}^2 + \rho\frobnorm{X-M}^2.
\end{equation}
Note, that due to the definition of~\eqref{eq:reghfull}, this can be seen as a convex-concave min-max
problem, by restricting the minimization of~$X$ over a compact set. By first solving for~$X$
one obtains,
\begin{equation}
    X = M + \fr{X_0-Z}{\rho} = \fr{(\rho+1)Y-Z}{\rho},
\end{equation}
where~$Y=\fr{X_0+\rho M}{1+\rho}$. Similarly, as in~\cite{larsson-olsson-ijcv-2016},
we get a program of the type (excluding constants)
\begin{equation}
\begin{aligned}
    &\max_Z\Bigg(\sum_{i=1}^n\min\left(b_i, \left[\sigma_i(Z)-a_i\right]_+^2\right)-\fr{\rho+1}{\rho}\frobnorm{Z-Y}^2 \\
    &\qquad+ \frobnorm{Z}^2 - \sum_{i=1}^n\left[\sigma_i(Z)-a_i\right]_+^2\Bigg).
\end{aligned}
\end{equation}
Again, the optimization can be reduced to the singular values only.
This bears strong resemblance to~\eqref{eq:unconstrained}, and
we show in the supplementary material that
Algorithm~\ref{algo:1} can be modified, with minimal effort, to solve
this problem as well.

\section{Experiments}\label{sec:experiments}
We demonstrate the shortcomings of using WNNM
for non-rigid reconstruction estimation and
structure-from-motion,
and show that our proposed method performs as good
or better than the current state-of-the-art.
In all applications, we apply the popular approach~\cite{candes2008enhancing,gu-2016,kumar-2019-arxiv}
to choose the weights inversely proportional to the singular values,
\begin{equation}
    w_i = \fr{C}{\sigma_i(X_0)+\epsilon},
\end{equation}
where~$\epsilon>0$ is a small number (to avoid division by zero),
and~$X_0$ is an initial estimate of the matrix~$X$. The trade-off parameter~$C$
will be tuned to the specific application. In the experiments, we use $w_i=2a_i$, and
choose~$b_i$ depending on the specific application. This allows us to control the
rank of the obtained solution
without excessive penalization of the non-zero singular values.

\subsection{Synthetic Missing Data}
In this section we consider the missing data problem with unknown rank
\begin{equation}
    \min_X\mu\rank(X)+\frobnorm{W\hada(X-M)}^2,
\end{equation}
where~$M$ is a measurement matrix, $\hada$ denotes the Hadamard (or element-wise) product,
and~$W$ is a missing data mask, with $w_{ij}=1$ if the entry $(i,j)$ is known, and zero otherwise.

Ground truth matrices~$M_0$ of size $32\times 512$ with $\rank(M_0)=4$ are generated, and to simulate
noise, a matrix~$N$ is added to obtain the
measurement matrix $M=M_0+N$. The entries of the noise matrix are normally
distributed with zero mean and standard deviation~$\sigma=0.1$.

When benchmarking image inpainting and deblurring, it is common to assume a uniformly distributed
missing data pattern. This assumption, however, is not applicable in many other subfields of
computer vision. In structure-from-motion the missing data pattern is typically very structured,
due to tracking failures. For comparison we show the reconstruction results for several methods,
on both uniformly random missing data patterns and tracking failures. The tracking failure
patterns were generated as in~\cite{larsson-olsson-cvpr-2017}.
The results are shown in Table~\ref{tab:synth}. Here we use the
$a_i = \fr{\sqrt{\mu}}{\sigma_i(M) + \epsilon}$, and
$b_i = \fr{\mu}{\sigma_i(M) + \epsilon}$, with $\epsilon=10^{-6}$. All other
parameters are set as proposed by the respective authors.

\subsection{Non-Rigid Deformation with Missing Data}
This experiment is constructed to highlight the downsides of using WNNM, and to illustrate
how shrinking bias can manifest itself in a real-world application.
Non-rigid deformations can be seen as a low-rank minimizing problem by assuming that
the tracked image points are moving in a low-dimensional subspace. This allows
us to model the points using a linear shape basis, where the complexity of the motion is
limited by the number of basis elements. This in turn, leads to the task of accurately making
trade-offs while enforcing a low (and unknown) rank, which leads to the problem formulation
\begin{equation}\label{eq:exp2}
    \min_{X}\mu\rank(X) + \norm{W\hada(X-M)},
\end{equation}
where $X=CB^\T$, with~$B$ being concatenated basis elements and~$C$ the corresponding
coefficient matrix. We use the experimental setup from~\cite{larsson-etal-eccv-2014},
where a KLT tracker is used on a video sequence. The usage of the tracker naturally induces
a structured missing data pattern, due to the inability to track the points through the entire
sequence.

We consider the relaxation of~\eqref{eq:exp2}
\begin{equation}
    \min_{X} \reg_h(X) + \frobnorm{W\hada(X-M)}^2,
\end{equation}
and choose $a_i = \fr{C}{\sigma_i(M)+\epsilon}$ and $b_i=0$ for $i\leq 3$ and $b_i=1/(C+\epsilon)$
otherwise. This choice of~$\vec{b}$ encourages a rank~3 solution without penalizing the large
singular values, and by choosing the parameter~$C$ vary the strength of the fix rang regularization
versus the weighted nuclear norm penalty. The datafit vs the parameter~$C$ is shown in
Table~\ref{tab:exp2}, and the reconstructed points for four frames of the \emph{book} sequence
are shown in Figure~\ref{fig:nonrigid2}.

Notice that, the despite the superior datafit for~$C=1$ (encouraging the WNNM
penalty), it is clear by visual inspection that the missing points are suboptimally recovered.
In Figure~\ref{fig:nonrigid2} the white center marker is the origin, and we note a tendency for
the WNNM penalty to favor solutions where the missing points are closer to the origin. This is the
consequence of a shrinking bias, and is only remedied by leaving the larger singular values
intact, thus excluding WNNM as a viable option for such applications.


\begin{table}[h]
\centering
\caption{Datafit for different values of~$C$. Note that the datafit for $C=1$ is better than
for~$C=10^{-2}$. This comes at the cost of incorrectly reconstructing the missing points,
as is shown in Figure~\ref{fig:nonrigid2}. The
datafit is measured as~$\frobnorm{W\hada(X-M)}$.}
\begin{tabular}{c|c c c c}
$C$ & $10^{-2}$ & 1 & 100   \\\hline
Datafit & 0.8354 & 0.4485 & 6.5221
\end{tabular}
\label{tab:exp2}
\end{table}


\begin{figure*}[t]
\centering
\def\w{3.8cm}
\setlength\tabcolsep{2pt} 
\begin{tabular}{cccc}
Frame 1 & Frame 121 & Frame 380 & Frame 668 \\
\includegraphics[width=\w]{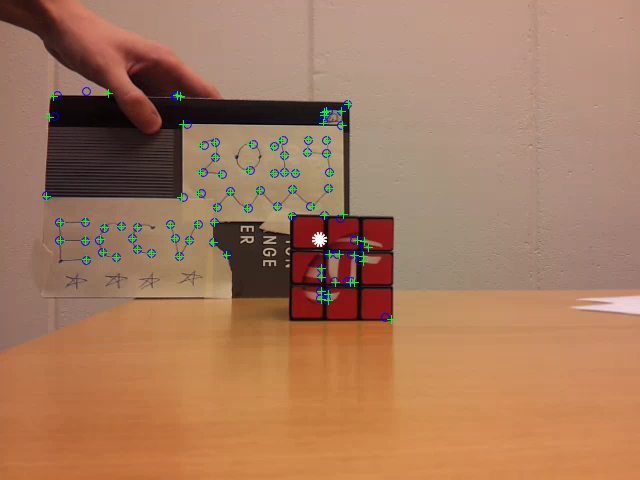} &
\includegraphics[width=\w]{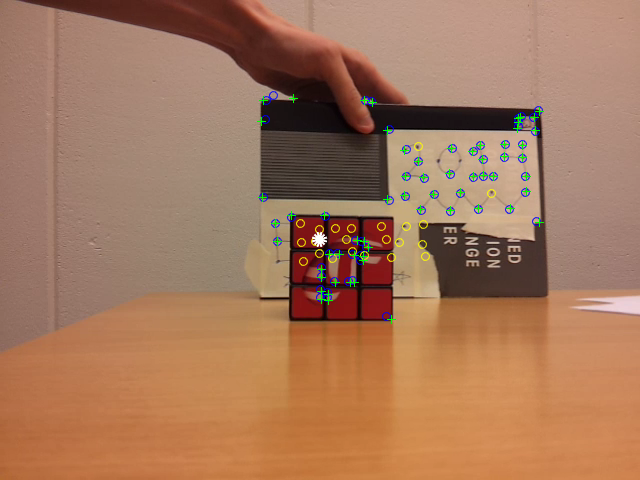} &
\includegraphics[width=\w]{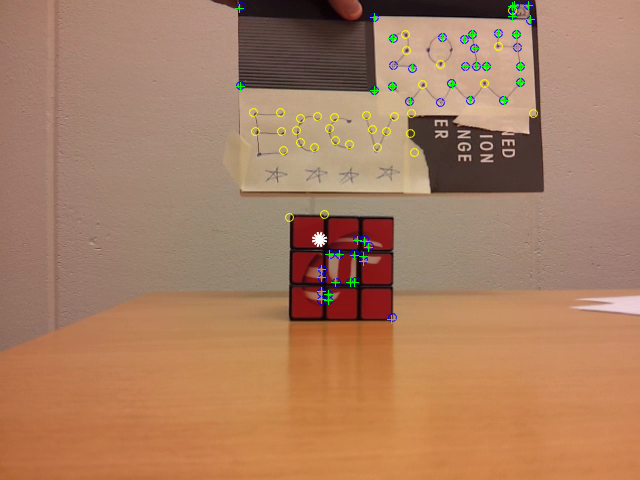} &
\includegraphics[width=\w]{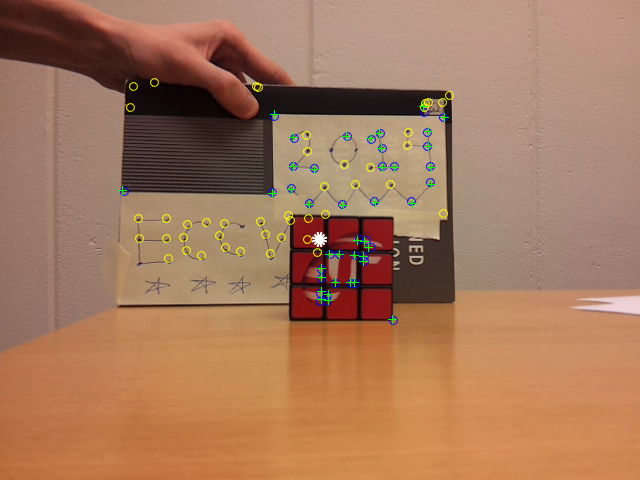} \\
\includegraphics[width=\w]{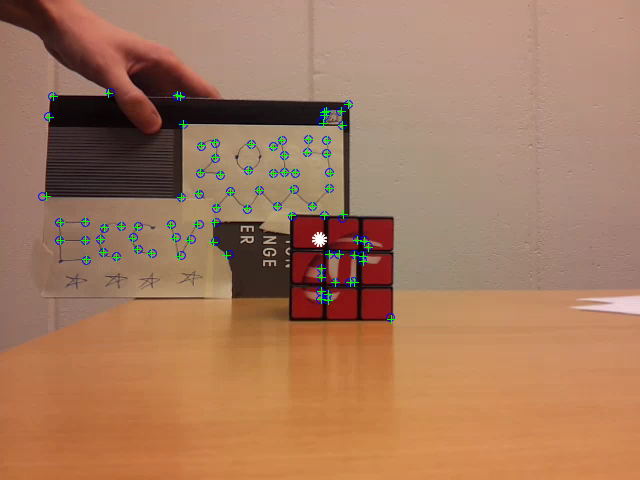} &
\includegraphics[width=\w]{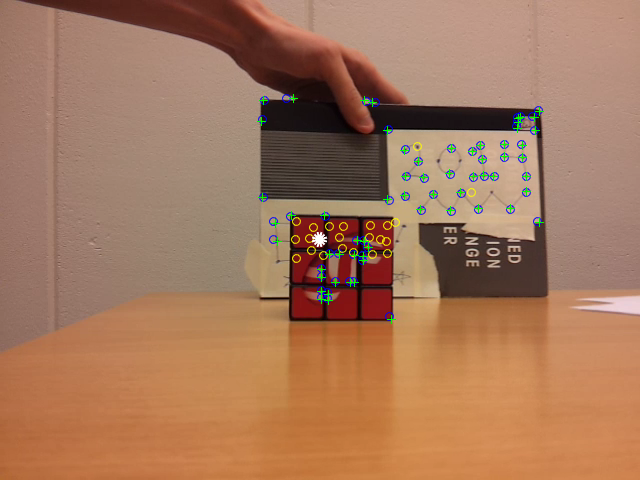} &
\includegraphics[width=\w]{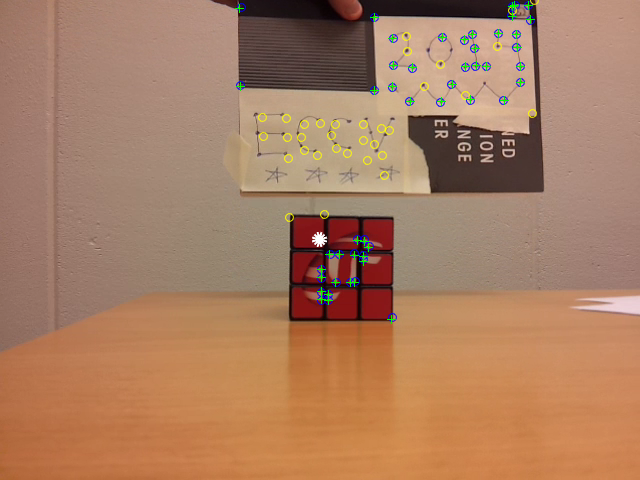} &
\includegraphics[width=\w]{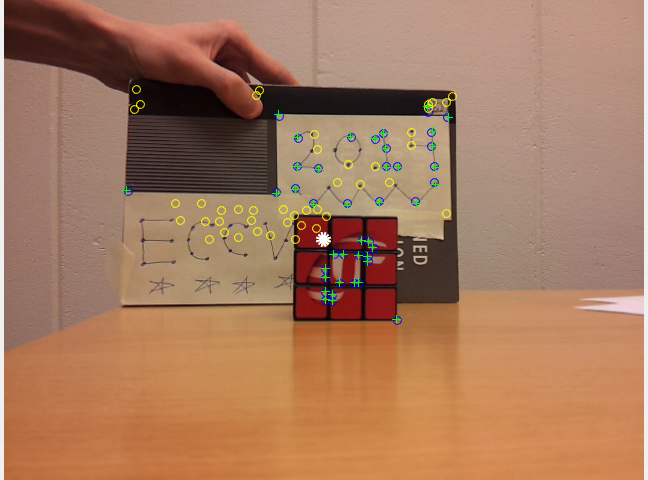} \\
\includegraphics[width=\w]{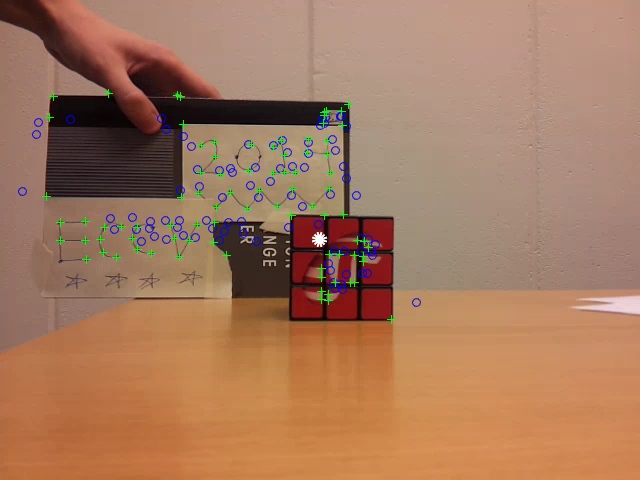} &
\includegraphics[width=\w]{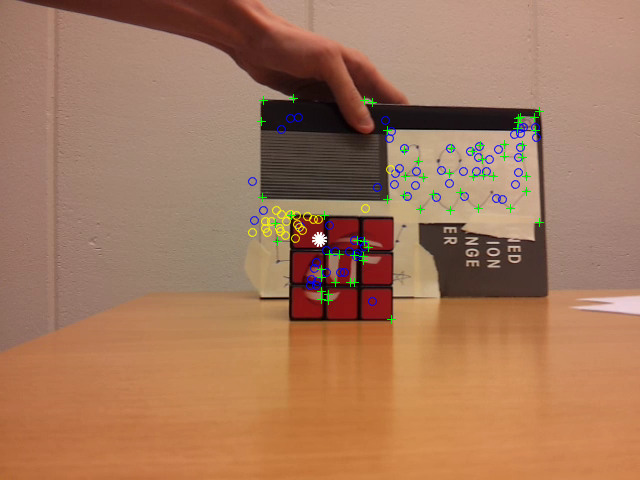} &
\includegraphics[width=\w]{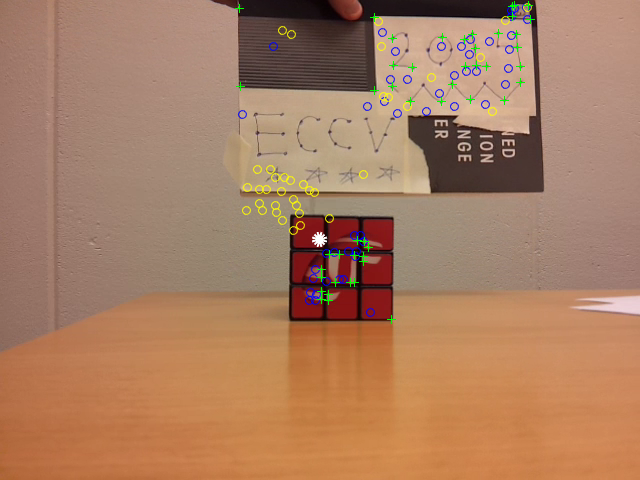} &
\includegraphics[width=\w]{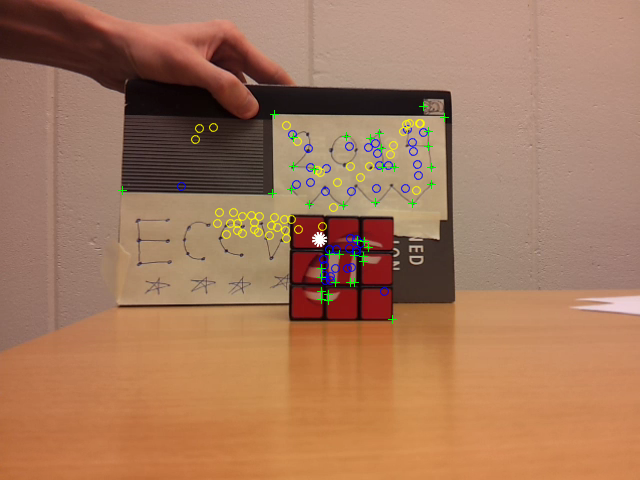} \\
\end{tabular}
\caption{From top to bottom, $C=10^{-2}$, $C=1$ and $C=100$. The white center dot is the origin
in the chosen coordinate system. The green crosses show the observed data, and the blue dots the reconstruction of these points. The yellow dots correspond to the recovered (and missing) data. Notice
the shrinking bias which is evident due to the recovered missing data being drawn towards the center
of the image as the WNNM penalty increases.}
\label{fig:nonrigid2}
\end{figure*}

\subsection{Motion Capture}
\label{sec:nrsfm}
The popular prior-free objective, proposed by Dai~\etal{}~\cite{dai-etal-ijcv-2014}, for NRSfM
\begin{equation}
    \min_X \mu\norm{X^\sharp}_* + \frobnorm{RX-M}^2,
\end{equation}
where~$X^\sharp$ a stacked version of~$X$ (see~\cite{dai-etal-ijcv-2014} for details),
suffers from shrinking bias, due to the nuclear norm penalty. Essentially, the nuclear norm
penalty is a way of relaxing the soft rank penalty, 
\begin{equation}
    \min_X \mu\rank(X^\sharp) + \frobnorm{RX-M}^2,
\end{equation}
however, it was shown in~\cite{olsson-etal-iccv-2017}, that simply using the convex
envelope of the rank function leads to non-physical reconstructions.
To tackle this situation, it was proposed to penalize
the 3D trajectories using a difference operator~$D$,
\begin{equation}
    \min_X \mu\rank(X^\sharp) + \frobnorm{RX-M}^2 + \frobnorm{DX^\sharp}^2.
\end{equation}
While such an objective leads to more physical solutions~\cite{olsson-etal-iccv-2017},
it also restricts the method
to ordered sequences of images. To allow for unordered sequences, we
suggest to replace the difference operator with an increasing penalty for smaller
singular values, modelled by an increasing sequence of weights~$\{a_i\}$. More specifically,
we consider the problem of minimizing
\begin{equation}\label{eq:our_nrsfm}
    \min_X \reg_h(X^\sharp) + \frobnorm{RX-M}^2,
\end{equation}
where sequences $\{a_i\}$ and $\{b_i\}$ are non-decreasing.
This bears resemblance to the weighted nuclear norm
approach presented in~\cite{kumar-2019-arxiv}, recently,
which coincide for the special case~$b_i\equiv 0$. Furthermore, this modified approach
exhibits far superior reconstruction results compared to the original method
proposed by Dai~\etal{}~\cite{dai-etal-ijcv-2014}.
In our comparison, we employ the same initialization heuristic for the
weights~$w_i$ on the singular values as in~\cite{gu-2016,kumar-2019-arxiv},
namely
\begin{equation}
\label{eq:kumar_weights}
    w_i = \fr{C}{\sigma_i(X_0^\sharp)+\epsilon},
\end{equation}
where $\epsilon=10^{-6}$  and $C>0$.
The matrix~$X_0^\sharp=R^+M$, where $R^+$ is the pseudo-inverse of~$R$, has successfully
been used as an
initialization scheme for NRSfM by others~\cite{dai-etal-ijcv-2014,valmadre-etal-2015,kumar-2019-arxiv}.

In practice, we choose $2a_i=w_i$, as in~\eqref{eq:kumar_weights},
with~$C=2\sqrt{\mu}$
and $b_i=w_i$, with $C=\mu$. This enforces mixed a soft-rank and hard rank thresholding.

We select four sequences from the CMU MOCAP dataset, and compare to the original method
proposed by Dai~\etal{}~\cite{dai-etal-ijcv-2014}, the newly proposed weighted
approach by Kumar~\cite{kumar-2019-arxiv}, the
method by Larsson and Olsson~\cite{larsson-olsson-ijcv-2016}
and our proposed objective~\eqref{eq:our_nrsfm},
all of which are prior-free, and do not assume that the image sequences are ordered.
For the nuclear norm approach by Dai~\etal{} we use the regularization
parameter~$\lambda=2\sqrt{\mu}$, and for Kumar, we set~$C=2\sqrt{\mu}$ (as for $\reg_h$)
and run the different
methods for a wide range of values for~$\mu$, using the same random initial solutions. We then
measure the datafit, defined as~$\frobnorm{RX-M}$ and the distance to ground
truth~$\frobnorm{X-X_{\text{gt}}}$, and show how these depend on the output rank (here defined
as the number of singular values larger than $10^{-6}$). By doing so, we see the ability
of the method
to make accurate trade-offs between fitting the data and enforcing the rank.
The results are shown in Figure~\ref{fig:mocapres}.

Note that, the datafit for all methods decrease with increase rank, which is to be expected;
however, we immediately note that the ``soft rank'' penalty~\eqref{eq:softrank}, in this case,
is too weak, which manifests itself by quickly fitting to data and the distance to ground truth
does not correlate with the datafit for solutions with rank larger than three.
For the revised method by Kumar~\cite{kumar-2019-arxiv}, as well as ours,
the correlation between the two quantities is much stronger.
What is interesting to see is that our method consistently performs better than the WNNM approach
for lower rank levels, suggesting that the shrinking bias is affecting the quality of these
reconstruction. Note, however, that the minimum distance to ground truth, obtained using the
WNNM approach is as good (or better) than the one obtained using~$\reg_h$. To obtain such a solution,
however, requires careful tuning of the~$\mu$ parameter and is unlikely to work on other datasets.

\begin{figure*}[p]
\centering
\setlength\tabcolsep{6pt} 
\def\w{0.365\linewidth}
\def\ww{0.19\linewidth}
\renewcommand{\arraystretch}{1.7}
\begin{tabular}{ccc}\
\emph{Drink} &
\multirow{2}{*}{\includegraphics[width=\w]{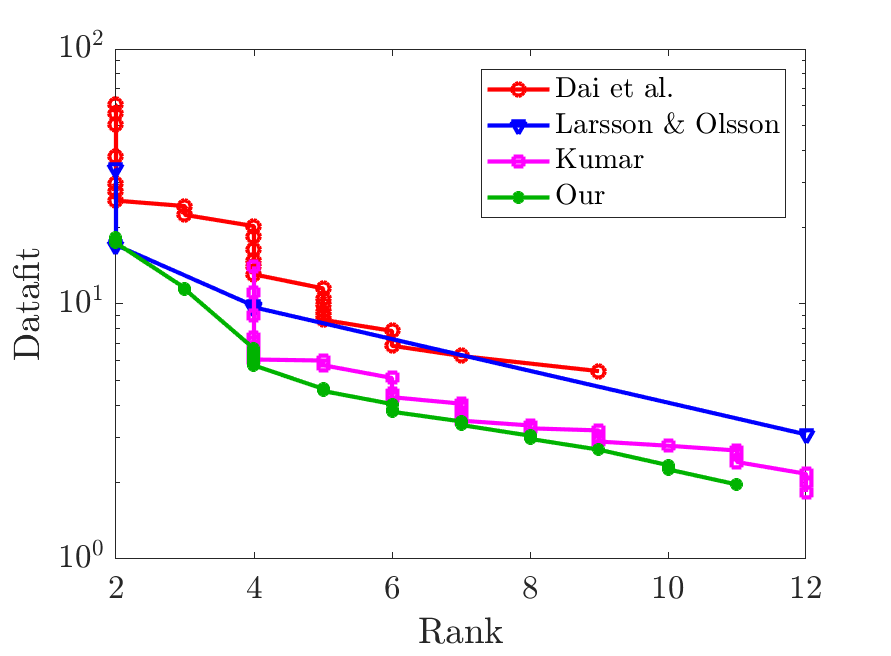}} &
\multirow{2}{*}{\includegraphics[width=\w]{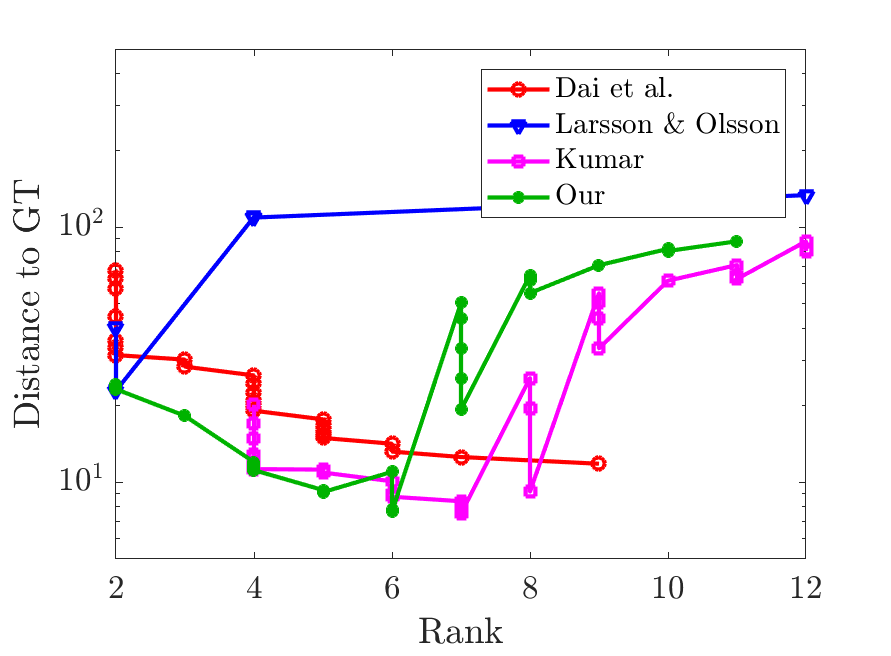}}\\
\includegraphics[width=\ww,trim={1.62cm 0 1.62cm 0},clip]{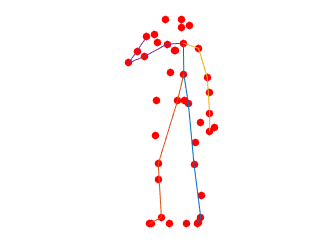} \\
\emph{Pickup} &
\multirow{2}{*}{\includegraphics[width=\w]{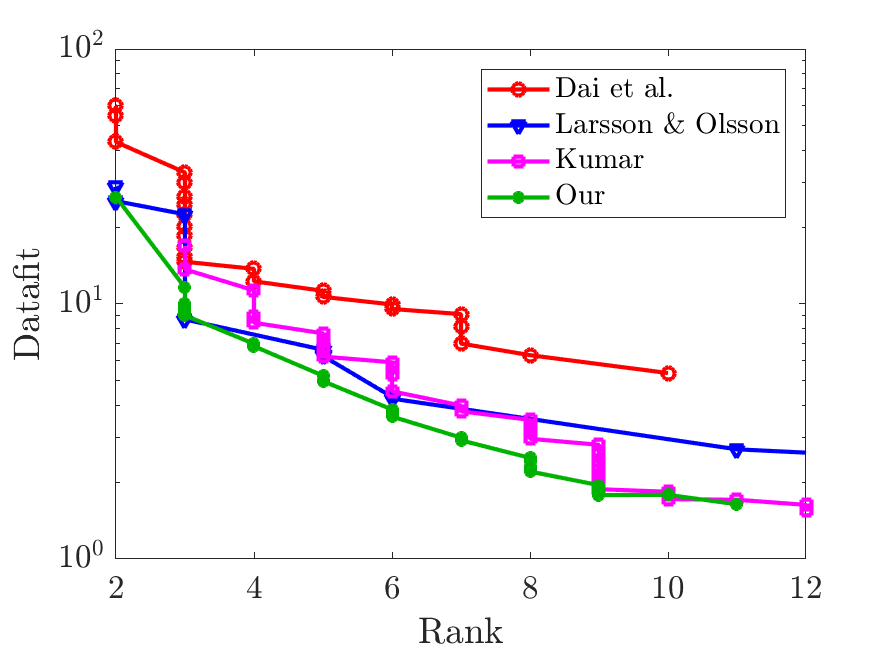}} &
\multirow{2}{*}{\includegraphics[width=\w]{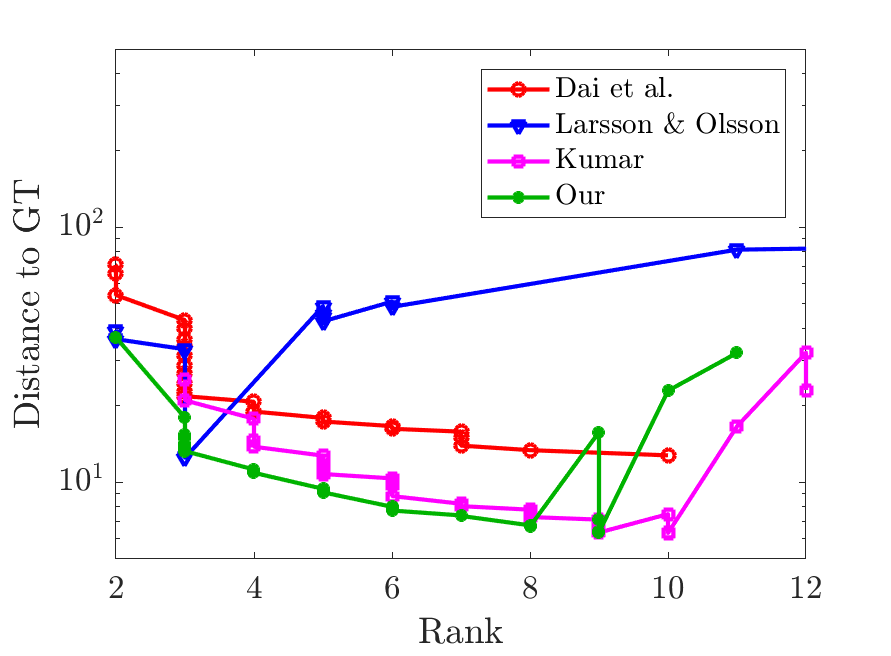}}\\
\includegraphics[width=\ww,trim={1.62cm 0 1.62cm 0},clip]{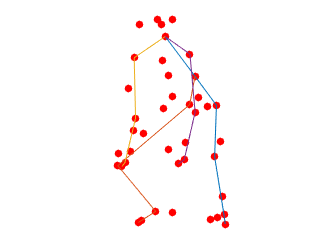} \\
\emph{Stretch} &
\multirow{2}{*}{\includegraphics[width=\w]{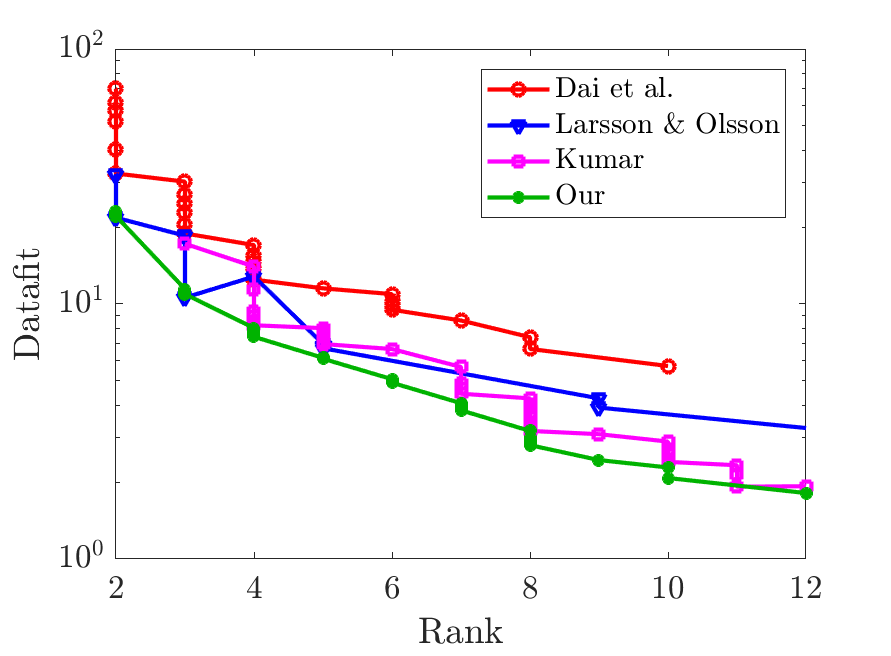}} &
\multirow{2}{*}{\includegraphics[width=\w]{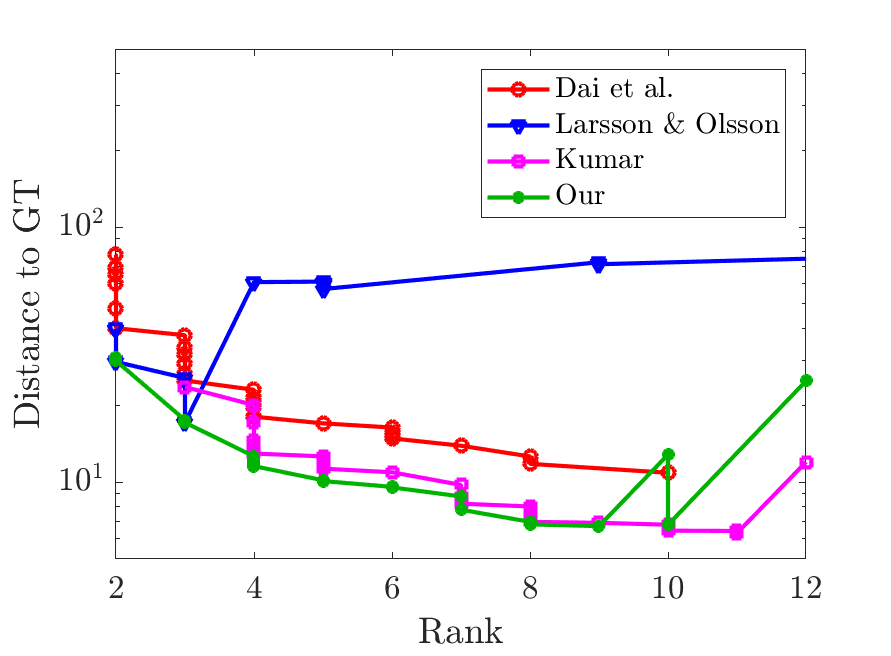}}\\
\includegraphics[width=\ww,trim={1.62cm 0 1.62cm 0},clip]{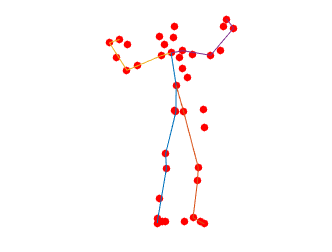} \\
\emph{Yoga} &
\multirow{2}{*}{\includegraphics[width=\w]{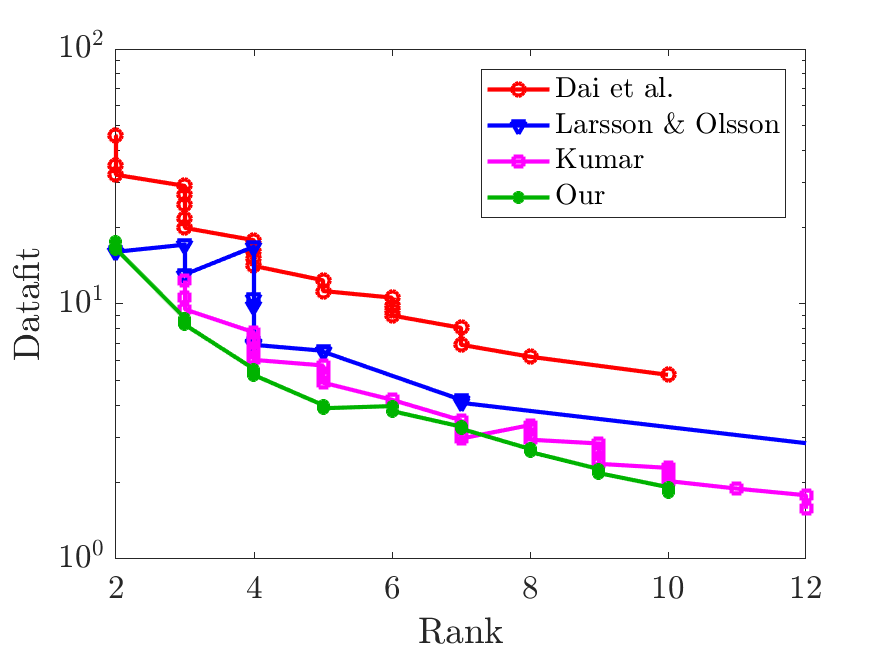}} &
\multirow{2}{*}{\includegraphics[width=\w]{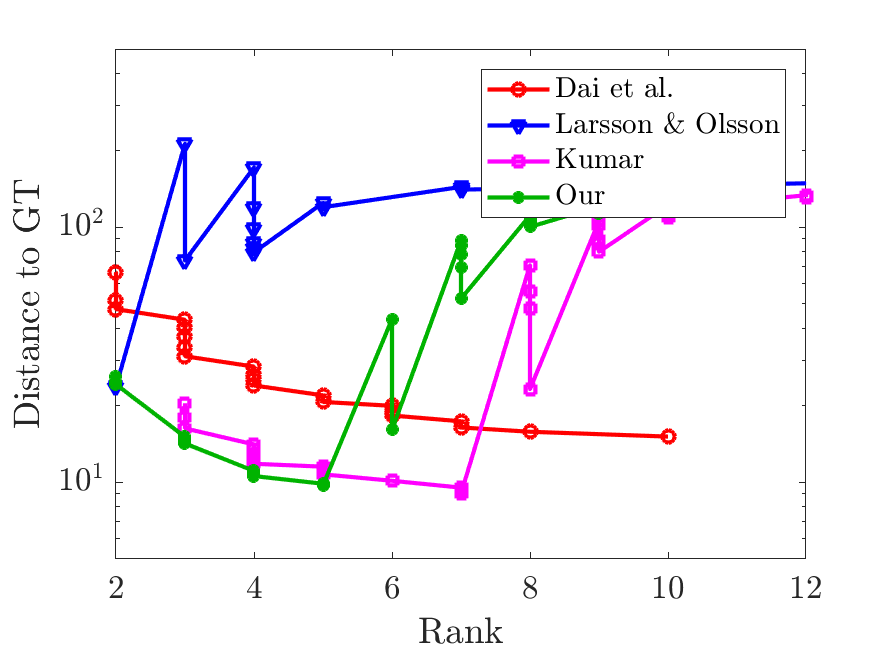}} \\
\includegraphics[width=\ww,trim={1.62cm 0 1.62cm 0},clip]{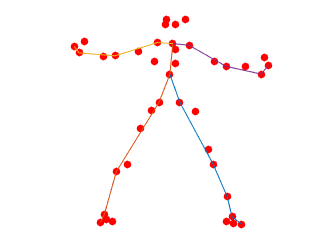} \\
\end{tabular}
\caption{Results for the experiment on the CMU MOCAP dataset.
\emph{First column:} Example images with skeleton added for visualization.
\emph{Second column:} The datafit, measured as~$\frobnorm{RX-M}$, as a function of the rank.
\emph{Last column:} Distance to ground truth, measured as~$\frobnorm{X-X_{\text{gt}}}$,
as a function of the rank.}
\label{fig:mocapres}
\end{figure*}

\section{Conclusions}
Despite success in many low-level imaging applications, there are
limitations of the applicability of WNNM in other applications of low-rank regularization.
In this paper, we have provided theoretical insight into the issues surrounding shrinking
bias, and proposed a solution where the shrinking bias can be partly or completely eliminated,
while keeping the rank low. This can be done using the proposed~$\reg_h$ regularizer,
which has the benefit of unifying weighted nuclear norm regularization with another class
of low-rank inducing penalties. Furthermore, an efficient
way of computing the regularizer has been proposed, as well as the related proximal
operator, which makes it suitable for optimization using splitting scheme, such as ADMM.

\clearpage
{\small
\bibliographystyle{ieee_fullname}
\bibliography{rh_admm}
}

\newpage
\twocolumn[
\vspace{0.8cm}
\begin{center}
{\Large \bf Supplementary Material:\\
A Unified Optimization Framework for Low-Rank Inducing Penalties\\}
\vspace{1.5cm}
\end{center}
]

\setcounter{section}{0} 

\begin{figure*}[t]
\centering
\includegraphics[width=0.382\linewidth]{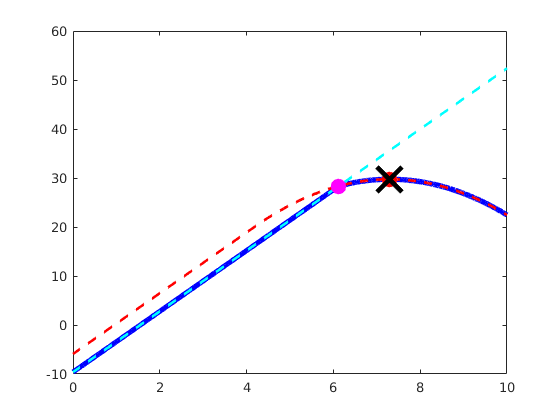}
\qquad
\includegraphics[width=0.382\linewidth]{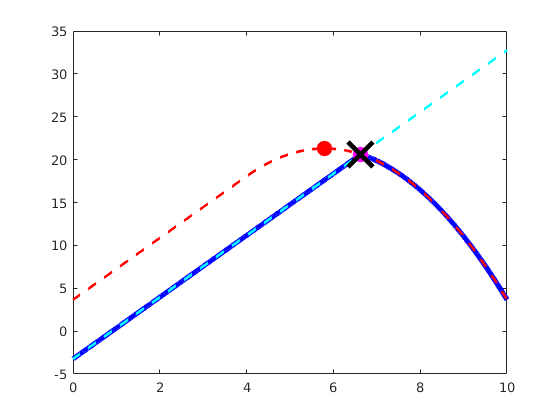}
\caption{Two different cases for values of $a_i$, $b_i$ and $\sigma_i(X)$
of~\eqref{eq:unconstrained}.}
\label{fig:illustration}
\end{figure*}

\section{The Fenchel Conjugate}
In this section, we compute the Fenchel conjugate of~\eqref{eq:fh},
which is necessary in order to find
the convex envelope.
Let $\inner{X}{Y}=\tr(X^{\T} Y)$, and note that we can write
\begin{equation}
\label{eq:Z}
    \inner{Y}{X}-\frobnorm{X-X_0} =
        \frobnorm{Z}^2 - \frobnorm{X_0}^2-\frobnorm{Z-X}^2.\\
\end{equation}
where $Z=\fr{1}{2}Y+X_0$.
By definition, the Fenchel conjugate of~\eqref{eq:fh} is given by
\begin{equation}
\begin{aligned}
    f^*_h(Y) &= \sup_X\inner{Y}{X}-f_h(X) \\
             &= \sup_X \frobnorm{Z}^2-\frobnorm{X_0}^2-\frobnorm{X-Z}^2-h(\sing(X)),
\end{aligned}
\end{equation}
where we use~\eqref{eq:Z} in the last step.
Note that the function~$h$, as well as the Frobenius norm, is unitarily invariant. Furthermore,
$\frobnorm{X-Z}^2=\frobnorm{X}^2+\frobnorm{Z}^2-2\inner{X}{Z}$, and $\inner{X}{Z}\leq\inner{\sing(X)}{\sing(Z)}$ by von Neumann's trace inequality, with equality if $X$ and $Z$ are
simultaneously unitarily diagonalizable. This reduces the problem to optimizing over
the singular values alone, which, after some manipulation, can be written as
\begin{equation}
\label{eq:fenchel1b}
\begin{aligned}
    f^*_h(Y)
        &= \max_{\sing(X)}
             -\frobnorm{X_0}^2 \\
               &\quad- \sum_{i=1}^k\left(\sigma_i^2(X) -  2[\sigma_i(Z)-a_i] \sigma_i(X) + b_i \right),
\end{aligned}
\end{equation}
where $\rank(X)=k$.
Considering each singular value separately leads to a program on the form
\begin{equation}
    \min_{x_i} x_i^2 - 2[\sigma_i(Z)-a_i] x_i +b_i,
\end{equation}
subject to $\sigma_{i+1}(X)\leq x_i \leq\sigma_{i-1}(X)$.
The sequence of unconstrained minimizers is given by $x_i=\sigma_i(Z)-a_i$.
If there exists $x_i<0$, then this is not the solution to the constrained
problem. Nevertheless, the sequence is non-increasing, hence there is an index $p$, such that
$x_p\geq 0$ and $x_{p+1}<0$\footnote{We allow the case~$p=0$, in which case the zero vector is optimal.}.

Note that
\begin{equation}
    \sum_{i=1}^kx_i^2-2s_ix_i = \norm{\vec{x}}^2-2\inner{\vec{x}}{\vec{s}},
\end{equation}
hence we can consider
optimizing $\norm{\vec{x}-\vec{s}}^2=\norm{\vec{x}}^2-2\inner{\vec{x}}{\vec{s}}+\norm{\vec{s}}^2$
subject to $x_1\geq x_2\geq\cdots\geq x_k\geq 0$.
Furthermore, $s_1\geq s_2\geq\cdots\geq s_k$.

Assume that minimum is obtained at $\vec{x}^*$ and fix $x_p^*$. Since $s_j<0$ for all $j>p$,
we must have $x^*_j=0$ for $j>p$. It is now clear that, $x_j^*=s_j$ otherwise,
hence $x_j^*=\max\{s_j,0\}=[s_j]_+$.
Inserting into~\eqref{eq:fenchel1b} gives
\begin{equation}
\label{eq:fenchel2case2}
    f^*_h(Y)
        = \max_{k}
             -\frobnorm{X_0}^2-
                \sum_{i=1}^k\left( b_i-[\sigma_i(Z)-a_i]_+^2 \right).
\end{equation}
Since $[s_i]_+=[\sigma_i(Z)-a_i]_+$ is non-increasing, and $b_i$ is non-decreasing, the maximizing~$k$
is obtained when
\begin{equation}
    [\sigma_k(Z)-a_k]_+^2\geq b_k
    \te{\, and\, }
    b_{k+1} \geq [\sigma_{k+1}(Z)-a_{k+1}]_+^2.
\end{equation}
For the maximizing $k=k^*$, we can write
\begin{equation}
\begin{aligned}
    &-\sum_{i=1}^{k^*}\left( b_i-[\sigma_i(Z)-a_i]_+^2 \right)\\
    &\quad= \sum_{i=1}^n[\sigma_i(Z)-a_i]_+^2 - \sum_{i=1}^n\min\{b_i,[\sigma_i(Z)-a_i]_+^2\}.
\end{aligned}
\end{equation}
From this observation, we get
\begin{equation}
\label{eq:case2fenchel}
\begin{aligned}
    f^*_h(Y) &= \sum_{i=1}^n\left[\sigma_i(\fr{1}{2}Y+X_0)-a_i\right]_+^2 - \frobnorm{X_0}^2 \\
             &-\sum_{i=1}^n\min\left(b_i, \left[\sigma_i(\fr{1}{2}Y+X_0)-a_i\right]_+^2\right).
\end{aligned}
\end{equation}

\section{The Convex Envelope}
Applying the definition of the bi-conjugate
$
    f^{**}_h(X) = \sup_Y\inner{Y}{X}-f^*_h(Y)
$ 
to~\eqref{eq:case2fenchel}, and introduce the change of variables $Z=\fr{1}{2}Y+X_0$
we get
\begin{equation}
\begin{aligned}
    f^{**}_h(X) &= \max_Z 2\inner{X}{Z-X_0}-\sum_{i=1}^n\left[\sigma_i(Z)-a_i\right]_+^2 \\
                &+\frobnorm{X_0}^2
    +\sum_{i=1}^n\min\left(b_i, \left[\sigma_i(Z)-a_i\right]_+^2\right).
\end{aligned}
\end{equation}
By expanding squares and simplifying, $2\inner{X}{Z-X_0}+\frobnorm{X_0}^2=\frobnorm{X-X_0}^2-\frobnorm{X-Z}^2+\frobnorm{Z}^2$,
which yields~\eqref{eq:reghfull}.

\section{Obtaining the Maximizing Sequences}
In this section we give the proof for the convergence of Algorithm~1, and how
to modify it to cope with the corresponding problem for the proximal operator.

\subsection{Proof of Theorem~\ref{thm:algo1}}
\begin{proof}[Proof of Theorem~\ref{thm:algo1}]
First, we will show that each step in the algorithm returns a solution to a constrained
subproblem~$P_i$, corresponding to a (partial) set of desired constraints~$\mathcal{Z}_i$.

Let~$P_0$ denote the unconstrained problem with solution~$\mathcal{s}\in\R_+^n$. Denote
the first interval generated in Algorithm~\ref{algo:1} by~$\iota_1=\{m_1,\ldots,n_1\}$,
and consider optimizing the first subproblem $P_1$
\begin{equation}
    \max_{z_{m_1}\geq \cdots\geq z_{n_1}} c(\vec{z}),
\end{equation}
where $\mathcal{Z}_1 = \{\vec{z}\in\mathcal{Z}_0\,|\,z_{m_1}\geq \cdots\geq z_{n_1}\}$.
By Lemma~\ref{lem:1} the solution vector is constant over the
subinterval~$z_i = s$ for $i\in\iota_1$, which is returned by the algorithm.
The next steps generates a solution
to subproblem of the form
\begin{equation}
    \max_{\scriptsize\begin{matrix}z_{m_1}\geq \cdots\geq z_{n_1}\\\vdots\\{z_{m_k}\geq \cdots\geq z_{n_k}}\end{matrix}} c(\vec{z}),
\end{equation}
corresponding to subproblem $P_k$. I the solution to subproblem~$P_k$ is in~$\mathcal{Z}$,
then it is a solution to the problem, otherwise one must add more constraints.
We solve problems on the form
\begin{equation}
    \max_{\vec{z}\in\mathcal{Z}_0} c(\vec{z})
    \geq
    \max_{\vec{z}\in\mathcal{Z}_1} c(\vec{z})
    \geq
    \cdots
    \geq
    \max_{\vec{z}\in\mathcal{Z}_\ell} c(\vec{z})
    =
    \max_{\vec{z}\in\mathcal{Z}} c(\vec{z}),
\end{equation}
where the last step yields a solution fulfilling the desired constraints.
Furthermore $\mathcal{Z}_0\supset\mathcal{Z}_1\supset\cdots\supset\mathcal{Z}_\ell\supset\mathcal{Z}$,
where~$\mathcal{Z} = \{\vec{z}\:|\:z_1\geq\cdots\geq z_n\geq 0\}$.
Finally, it is easy to see that the algorithm terminates, since there are only finitely
many possible subintervals.
\end{proof}

\begin{figure*}[t]
\centering
\includegraphics[width=0.3\linewidth]{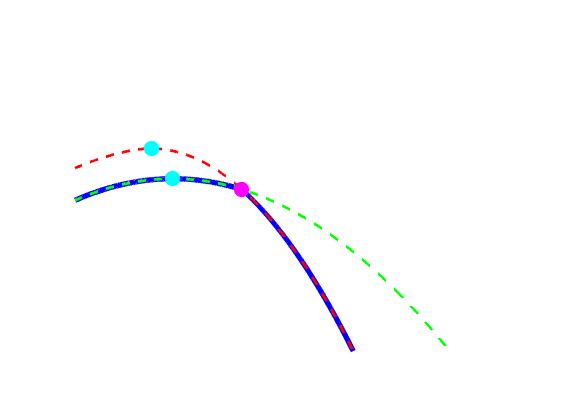}
\quad
\includegraphics[width=0.3\linewidth]{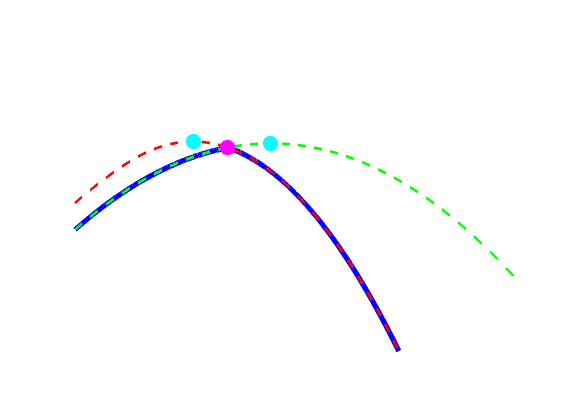}
\quad
\includegraphics[width=0.3\linewidth]{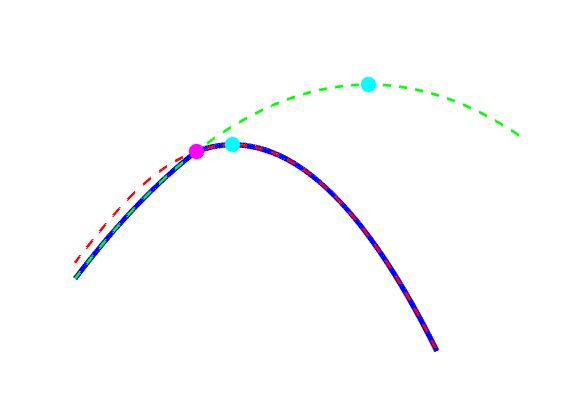}
\caption{Illustration of the three different cases for the proximal operator.}
\label{fig:prox}
\end{figure*}

\subsection{Modifying Algorithm~\ref{algo:1}}
Following the approach used in~\cite{larsson-olsson-ijcv-2016},
consider the program
\begin{equation}
\label{eq:unconstrained2}
\begin{aligned}
&\max_s    && \min\{b_i,[s-a_i]_+^2\}-\fr{\rho+1}{\rho}(s-\sigma_i(Y))^2 \\
&&&+s^2-[s-a_i]_+^2, \\
&\te{s.t.} && \sigma_{i+1}(Z)\leq s\leq \sigma_{i-1}(Z).
\end{aligned}
\end{equation}
Note that the objective function is the pointwise minimum of
\begin{equation}
\begin{aligned}
f_1(s) &= b_i-\fr{\rho+1}{\rho}(s-\sigma_i(Y))^2+s^2-[s-a_i]_+^2,\\
f_2(s) &= s^2-\fr{\rho+1}{\rho}(s-\sigma_i(Y))^2,
\end{aligned}
\end{equation}
both of which are concave, since~$\fr{\rho+1}{\rho}>1$.
For $f_1$ the maximum is obtained in $s=\fr{a_i\rho}{\rho+1}+\sigma_i(Y)$, if $s\geq a_i$ otherwise
when~$s=(\rho+1)\sigma_i(Y)$. The minimum of $f_2$ is obtained
when~$s=(\rho+1)\sigma_i(Y)$.

There are three possible cases, also shown in Figure~\ref{fig:prox}.
\begin{enumerate}
\item The maximum occurs when $s>a_i+\sqrt{b_i}$, hence $f_1(s)<f_2(s)$,
hence~$s=\fr{a_i\rho}{\rho+1}+\sigma_i(Y)$.
\item The maximum occurs when $s<a_i+\sqrt{b_i}$, where $f_1(s)>f_2(s)$,
hence~$s=(\rho+1)\sigma_i(Y)$.
\item When $s=a_i+\sqrt{b_i}$, which is valid elsewhere.
\end{enumerate}
in summary
\begin{equation}\label{eq:si2}
s_i = \left\{
\begin{aligned}
    &\fr{a_i\rho}{\rho+1}+\sigma_i(Y), && \fr{a_i}{\rho+1}+\sqrt{b_i}< \sigma_i(Y), \\
    &a_i+\sqrt{b_i}, && \fr{a_i+\sqrt{b_i}}{1+\rho}\leq \sigma_i(Y) \leq \fr{a_i}{\rho+1}+\sqrt{b_i},\\
    &(1+\rho)\sigma_i(Y), && \sigma_i(Y)<\fr{a_i+\sqrt{b_i}}{1+\rho},
\end{aligned}
\right.
\end{equation}
By replacing the sequence of unconstrained minimizers~$\{s_i\}$ defined by~\eqref{eq:si2},
with the corresponding sequence in Section~\ref{sec:finding} (of the main paper),
and changing the objective function of Algorithm~\ref{algo:1},
to the one in~\eqref{eq:unconstrained2}, the maximizing singular value vector fo the proximal
operator is obtained.

\end{document}